\documentclass[11pt,twoside]{article}

\usepackage{geometry}
\usepackage{setspace}
\usepackage{fancyhdr}
\usepackage{hyperref}

\usepackage{amsmath,amsthm,amssymb,amsfonts,personal}
\usepackage{mathtools}
\usepackage[mathscr]{euscript}
\usepackage{bm}
\usepackage{subcaption}
\usepackage{color}
\newtheorem{assumption}{Assumption}

\newtheorem{theorem}{Theorem}
\theoremstyle{remark}

\makeatletter
\newtheoremstyle{mytheoremstyle}%
{\topsep}{\topsep}
{}{}            
{}{}            
{0.5em}
{\thmname{\@ifempty{#3}{#1}\@ifnotempty{#3}{#3}}}
\makeatother
\theoremstyle{mytheoremstyle}

\allowdisplaybreaks

\usepackage{enumitem} 
\usepackage{caption} 
\usepackage{booktabs} 
\newsavebox\CBox

\usepackage{natbib}
 \bibpunct[, ]{(}{)}{,}{a}{}{,}%
 \def\newblock{\ }%
\DeclareSymbolFont{matha}{OML}{txmi}{m}{it}
\DeclareMathSymbol{\varv}{\mathord}{matha}{118}

\makeatletter
\newcommand{\vast}{\bBigg@{4}}
\newcommand{\Vast}{\bBigg@{5}}   
\makeatother

\def\E{\mathbb{E}}

\begin{document}

\title{\vspace{-1em}
Approximating Systems Fed by Poisson Processes with Rapidly Changing Arrival Rates\vspace{-0.6em}}

\author{
	{\normalsize Zeyu Zheng} \vspace{-0.6em}\\
	{\footnotesize Department of Management Science and Engineering} \vspace{-0.8em}\\
	{\footnotesize Stanford University, zyzheng@stanford.edu}\\
	{\normalsize Harsha Honnappa} \vspace{-0.6em}\\
	{\footnotesize School of Industrial Engineering} \vspace{-0.8em}\\
	{\footnotesize Purdue University, honnappa@purdue.edu}\\
	{\normalsize Peter W. Glynn} \vspace{-0.6em}\\
	{\footnotesize Department of Management Science and Engineering} \vspace{-0.8em}\\
	{\footnotesize Stanford University, glynn@stanford.edu}\\
}

\date{July 17, 2018}

\maketitle

\pagestyle{fancy}
\fancyhead[RO,LE]{\small \thepage}
\fancyfoot[C]{}

\begin{abstract} 
\onehalfspacing
\noindent
This paper introduces a new asymptotic regime for simplifying stochastic models having non-stationary effects, such as those that arise in the presence of time-of-day effects.  This regime describes an operating environment within which the arrival process to a service system has an arrival intensity that is fluctuating rapidly. We show that such a service system is well approximated by the corresponding model in which the arrival process is Poisson with a constant arrival rate. In addition to the basic weak convergence theorem, we also establish a first order correction for the distribution of the cumulative number of arrivals over $[0,t]$, as well as the number-in-system process for an infinite-server queue fed by an arrival process having a rapidly changing arrival rate. This new asymptotic regime provides a second regime within which non-stationary stochastic models can be reasonably approximated by a process with stationary dynamics, thereby complementing the previously studied setting within which rates vary slowly in time.

\end{abstract}

\vspace{-1em}
\noindent\makebox[\linewidth]{\rule{\linewidth}{0.8pt}}


\section{Introduction}
In many operations management settings, the arrival process to the system exhibits clear non-stationarities. These non-stationarities may arise as a consequence of time-of-day effects, day-of-week effects, seasonalities, or stochastic fluctuations in the arrival rate. One mathematical vehicle for studying such non-stationary arrival processes is to consider the setting in which the arrival rate changes slowly in time. In this setting, it is intuitively clear that the non-stationary system can be viewed as a small perturbation of a constant arrival rate system. Consequently, it seems conceptually reasonable that one should be able to study such slowly changing arrival rate models via an asymptotic expansion in which each of the terms in the expansion involve a stationary arrival rate calculation. This intuition has been validated rigorously by \cite*{khashinskii1996asymptotic}, \cite*{massey1998uniform}, and, more recently, by \cite*{zheng2018approximating}.

In this paper, we show that arrival rate modeling also simplifies significantly at the opposite end of the asymptotic spectrum in which the arrival rates fluctuate rapidly. Thus, we can view the results of this paper as complementing the existing literature on slowly varying arrival rate modeling. In particular, we study systems fed by Poisson processes in which the intensity at time $t$ is given by $\lambda(t/\epsilon)$, where $\epsilon$ is a small parameter and $\lambda = (\lambda(s):s\ge 0)$ is a fixed process. The process $\lambda$ could be a deterministic periodic function, or it could be a functional of a positive recurrent Markov process. In either case, we show that when $\epsilon\downarrow 0$, we may view the system as one fed by a constant rate Poisson process with rate $\lambda^*$ given by the long-run time-average of $\lambda$; see Theorem 1 for details. Thus, this paper provides a second rigorously supported asymptotic regime within which the dynamics of a service system with a non-stationary arrival process can be approximated by a simpler system with stationary dynamics. We note that despite the practical importance of such non-stationary models, very few analytical approximations are available for such systems. 

These high frequency fluctuations in the arrival rate may be a consequence of a short period, stochastic effects, or some combination of high frequency periodicity and rapid stochastic fluctuations. As an example of a real-world system in which such an asymptotic regime may be appropriate, consider a construction equipment leasing company. If the leases tend to be of long duration (e.g., on the order of months), our theory suggests that in the analysis of a queueing model intended to predict lost sales (due to all the available equipment having been rented), one can safely ignore the daily periodicity in the arrival rate describing exogenous demand for the company's equipment.

This note is organized as follows. Section 2 provides our main weak convergence theorem, establishing that point processes with rapidly fluctuating intensities can be weakly approximated by a constant rate Poisson process (Theorem 1). In the remainder of the section, we compute the total variation (tv) distance between the point process and the Poisson process in the Markov-modulated doubt stochastic setting, and prove that the tv distance does not tend to zero, thereby showing that one can expect to use the constant rate Poisson approximation only for suitably continuous path functionals. In Section 3, we study the distribution of the total number of arrivals in an interval [0,t], and obtain a first order refinement to the weak convergence theorem that reflects the first order impact of the high frequency fluctuations in the arrival rate; see Theorem 3. Finally, Section 4 provides a similar first order refinement in the setting of the number-in-system process for the infinite-server queue; see Theorem 4. 

\section{Weak Convergence to a Constant Rate Poisson Process}
To construct our point process with a rapidly fluctuating arrival rate, we start with a fixed arrival counting process $N=(N(t):t \ge 0)$. We assume that $N$ is \textit{simple}, in the sense that $N$ increases exactly by one at each arrival epoch (and hence no batch arrivals are possible). We further require that $N$ be adapted to a filtration $\mathcal{F} = (\mathcal{F}_t: t\ge 0 )$, and that $N$ posessesses a right continuous non-decreasing $\mathcal{F}$-compensator $A=(A(t):t\ge 0)$, so that $M = (M(t):t\ge 0)$ is a martingale adapted to $\mathcal{F},$ where 
\[
M(t)=N(t) - A(t).
\]
Note that $N$ need not be a doubly stochastic Poisson process (e.g. $N$ could be a Hawkes process; see \cite{hawkes1971spectra} for the definition). 

For $1\ge \epsilon > 0$, let $\beta^\epsilon = (\beta_i^\epsilon ;\, i \ge 1)$ be an independent and identically distributed (iid) sequence of Bernoulli($\epsilon$) random variables (rv's) independent of $N$. For $t\ge 0$, let $A_\epsilon(t) = \epsilon A(t/\epsilon)$, and let $\mathcal{G}_t^\epsilon$ be the smallest $\sigma$-algebra containing $\mathcal{F}_{t/\epsilon}$ and the $\sigma$-algebra $\sigma(\beta_i^\epsilon: 1\le i\le N(t/\epsilon))$. Put $\mathcal{G}^\epsilon = (\mathcal{G}_t^\epsilon: t\ge 0)$ and 
\[
N_\epsilon(t) = \sum_{i=1}^{N(t/\epsilon)}\beta_i^\epsilon.
\] 
Then, $N_\epsilon = (N_\epsilon(t): t\ge 0)$ is a simple point process for which 
\begin{align*}
&\E [N_\epsilon(t+s) - A_\epsilon(t+s)|\mathcal{G}_t^\epsilon]\\
=\,& N_\epsilon(t) + \E\left[ \sum_{i=N(t/\epsilon)+1}^{N((t+s)/\epsilon)}\beta_i^\epsilon\Big|\mathcal{G}_t^\epsilon \right] - \E[A_\epsilon(t+s)|\mathcal{G}_t^\epsilon]\\
=\,& N_\epsilon(t)+ \E \beta_t^\epsilon\E[(N((t+s)/\epsilon) - N(t/\epsilon)|\mathcal{G}_t^\epsilon] - \epsilon\E[A((t+s)/\epsilon)|\mathcal{G}_t^\epsilon]\\
=\,& N_\epsilon(t) + \epsilon\E[N((t+s)/\epsilon)-A((t+s)/\epsilon)|\mathcal{G}_t^\epsilon] - \epsilon N(t/\epsilon)\\
=\,& N_\epsilon(t) + \epsilon\E[N((t+s)/\epsilon)-A((t+s)/\epsilon)|\mathcal{F}_{t/\epsilon}] - \epsilon N(t/\epsilon)\\
=\,& N_\epsilon(t) + \epsilon(N(t/\epsilon) - A(t/\epsilon)) - \epsilon N(t/\epsilon)\\
=\,& N_\epsilon(t) - A_\epsilon(t)
\end{align*}
for $s,t\ge 0$, so that $A_\epsilon$ is the $\mathcal{G}^\epsilon$-compensator of $N_\epsilon$. (Here, we used the independence of $\beta^\epsilon$ from $N$ in the third last equality (see p.87 of \cite*{kallenberg1997foundations}), and the fact that $M$ is an $\mathcal{F}$-adapted martingale in the second last equality.)

An important special case is when the compensator $A$ can be written in the form 
\[
A(t) = \int_{0}^{t}\lambda(s)ds,
\]
in which case $\lambda=(\lambda(t):t\ge 0)$ is the $\mathcal{F}$-intensity of $N$. Then, $N_\epsilon$ has $\mathcal{G}^\epsilon$-intensity $\lambda_\epsilon = (\lambda_\epsilon(t):t\ge 0)$, where $\lambda_\epsilon(t) = \lambda(t/\epsilon).$ We can see clearly, in this setting, that $N_\epsilon$ has a rapidly fluctuating intensity as $\epsilon\downarrow 0$, so that this framework is indeed modeling such an asymptotic regime. 

We now assume: 
\begin{assumption}
	There exists a deterministic $\lambda^* \in (0,\infty)$ such that \[
	\frac{1}{t}A(t)\Rightarrow \lambda^*
	\]
	as $t\rightarrow\infty,$ where $\Rightarrow$ denotes weak convergence. 
\end{assumption}
Here is our main result of this section. Recall that $D[0,\infty)$ is the space of right continuous functions on $[0,\infty)$ having left limits, endowed with the Skorohod $J_1$ topology; see \cite{ethier1986markov} for details. 
\begin{theorem}
	In the presence of Assumption 1, 
	\[
	N_\epsilon\Rightarrow N_0
	\]
	in $D[0,\infty)$ as $\epsilon\downarrow 0$, where $N_0 = (N_0(t):t\ge 0)$ is a Poisson process with constant intensity $\lambda^*$. 
\end{theorem}
	\begin{proof}
	We note that Assumption 1 implies that for each $t\ge 0$, 
	\[
	A_\epsilon(t) = \epsilon A(t/\epsilon) = \left(\frac{\epsilon}{t}\right) A(t/\epsilon) \cdot t\Rightarrow \lambda^* t
	\]
	as $\epsilon\downarrow 0$. We now apply Theorem 13.4.IV of \cite*{daley1988} to obtain the result.
\end{proof}

Of course, arrival processes typically serve as models describing exogenous inputs to queueing systems or service systems. Other sources of randomness described (say) by a random sequence (such as service time requirements, abandonment times, etc) will typically also be present. If $Z$ is independent of $N_\epsilon$, it follows from Theorem 1 that 
\[
(Z,N_\epsilon)\Rightarrow (Z,N_0)
\]
in $\mathbb{R}^\infty \times D[0,\infty)$ as $\epsilon\downarrow 0$. It follows that if $h:\mathbb{R}^\infty \times D[0,\infty)\rightarrow \mathbb{R}$ is continuous in the product topology at $(Z,N_0)$ a.s., then 
\[
h(Z,N_\epsilon)\Rightarrow h(Z,N_0)
\]
as $\epsilon\downarrow 0$ (via the continuous mapping principle; see \cite{billingsley1968convergence}, p.21). 

Consequently, if $h$ is a map that sends $(Z,N_\epsilon)$ into some associated performance measure (e.g. the number-in-system at time $t$), we may infer that the performance measure can be computed as if the point process $N_\epsilon$ is Poisson with rate $\lambda^*$ (when $\epsilon$ is small). 

In the remainder of this section, we make clear that while $N_\epsilon$ converges weakly to $N_0$ in $D[0,\infty)$ as $\epsilon\downarrow 0$, no convergence typically takes place in the total variation norm. More specifically, suppose that $N_\epsilon$ is a doubly stochastic Poisson process with stochastic intensity $\lambda_\epsilon = (\lambda_\epsilon (t): t\ge 0)$, where $\lambda_\epsilon(t) = \lambda(t/\epsilon)$ for some fixed intensity $\lambda.$ Suppose that $S$ is a complete separable metric space. Recall that an $S$-valued Markov process $X=(X(t):t\ge 0)$ is said to be \textit{$v$-geometrically ergodic} if there exists a (measurable) function $v\ge 1$, a probability $\pi$ on $S$, $d<\infty$, and $\alpha>0$ such that 
\begin{equation}
\sup_{|g|\le v} \Big|\E_x g(X(t))-\int_S g(y)\pi(dy)\Big| \le d\, v(x) e^{-\alpha t} \label{eq:2.1}
\end{equation}
for $t\ge 0$ and $x\in S$, where $\E_x(\cdot)\triangleq \E(\,\cdot\,|X(0)=x)$; see Down, Meyn, and Tweedie (1995) for sufficient conditions assuring such geometric ergodicity. 

We assume that:
\begin{assumption}
	$\lambda(t) = f(X(t))$ for some bounded continuous $f:S\rightarrow \mathbb{R}_+$, where $X$ is $v$-geometrically ergodic. 
\end{assumption}

To state our next result on the total variation distance between $N_\epsilon$ and $N_0$, we let $X_1(\infty), X_2(\infty),\ldots$ be an iid sequence of $S$-valued rv's having common distribution $\pi$ (independent of $N_0$). 
\begin{theorem}
	Suppose Assumption 2 holds and $\E f(X_1(\infty))>0$. Then, \begin{align*}
	\sup_A |P((N_\epsilon(s):0\le s\le t)\in A) - P((N_0(s):0\le s\le t)\in A)|\rightarrow\frac{1}{2}\E\left|\prod_{j=1}^{N_0(t)}\frac{f(X_j(\infty))}{\E f(X_1(\infty))} -1\right|
	\end{align*}
	as $\epsilon\downarrow 0$, where the supremum is taken over the Borel subsets of $D[0,t]$. 
\end{theorem}
\begin{proof}
	The change-of-measure formula for doubly stochastic Poisson processes (see, for example, p.241 of \cite*{bremaud1981point}) asserts that 
	\begin{align*}
	P((N_\epsilon(s):0\le s\le t)\in A) = \E I((N_0(s): 0\le s\le t)\in A)\exp(-\int_{0}^{t}\tilde{\lambda}_\epsilon(s)ds )\cdot \prod_{j=1}^{N_0(t)} \left(\frac{\lambda_\epsilon (T_j)}{\lambda^*}\right),
	\end{align*}
	where $T_1,T_2,\ldots$ are the consecutive jump times of $N_0$, $\lambda^* = \E f(X_1(\infty))$, $N_0$ is a Poisson process with constant rate $\lambda^*$ under $P$, and $\tilde{\lambda}_\epsilon (s) = \lambda_\epsilon(s) - \lambda^*$. It follows that (see, for example, \cite*{gibbs2002choosing})
	\begin{align}
	&\sup_A|P((N_\epsilon(s):0\le s\le t)\in A) - P((N_0(s):0\le s\le t)\in A)|\nonumber \\
	&= \frac{1}{2}\E |\exp(-\int_{0}^{t}\tilde{\lambda}_\epsilon(s)ds) \prod_{j=1}^{N_0(t)}{\left(\frac{\lambda_\epsilon (T_j)}{\lambda^*}\right)}-1|. \label{eq:2.2}
	\end{align}
	Let $\mathcal{H}$ be the $\sigma$-algebra generated by $T_1,T_2,\ldots,T_{N_0(t)},N_0(t)$. Conditional on $\mathcal{H}$, Assumption 2 implies that 
	\begin{align*}
	&P(\lambda_\epsilon(T_i)\le x_i,\, 1\le i\le N_0(t)\,|\, \mathcal{H} )\\ 
	=\,& \E( I(f(X(T_i/\epsilon))\le x_i, \, 1\le i\le N_0(t)-1) P(f(X({T_{N_0(t)}/\epsilon	}))\le x_{N_0(t)}\,|\,  X(T_{N_0(t)-1}/\epsilon))  \,|\, \mathcal{H}).
	\end{align*}
	Since $I(f(\,\cdot\,)\le y)$ is upper bounded by $v$, Assumption 2 ensures that 
	\[
	p_\epsilon(s,x,y) \triangleq P(f(X(s/\epsilon))\le y\,|\,X(0)=x) \rightarrow P(f(X(\infty))\le y)
	\]
	as $\epsilon\downarrow 0$, so that 
	\begin{align*}
	&|\, P(\lambda_\epsilon(T_i)\le x_i,\, 1\le i\le N_0(t)\,|\, \mathcal{H} )  - P(\lambda_\epsilon(T_i)\le x_i,\, 1\le i\le N_0(t)-1)\,|\, \mathcal{H} )P(f(X(\infty))\le x_{N_0(t)})\,|\\
	=\, &|\,  \E( I(\lambda_\epsilon(T_i))\le x_i, \, 1\le i\le N_0(t)-1) (p_\epsilon(T_{N_0(t)} - T_{N_0(t)-1}, X(T_{N_0(t)-1}/\epsilon),x_{N_0(t)} ) - P(f(X(\infty))\le x_{N_0(t)}))|\mathcal{H})|\\
\le、，	& \E |\, p_\epsilon(T_{N_0(t)} - T_{N_0(t)-1}, X(T_{N_0(t)-1}/\epsilon),x_{N_0(t)} ) - P(f(X(\infty))\le x_{N_0(t)}) \,|\rightarrow 0
	\end{align*}
	as $\epsilon\downarrow 0.$ We now repeat this argument $N_0(t)-1$ additional times, thereby yielding 
	\begin{align*}
	P(\lambda_\epsilon(T_i)\le x_i,\, 1\le i\le N_0(t)\,|\, T_1,T_2,\ldots,T_{N_0(t)},N_0(t) )\rightarrow\, \prod_{i=1}^{N_0(t)}P(f(X_i(\infty))\le x_i)
	\end{align*}
	as $\epsilon\downarrow0$. Hence, conditional on $T_1,\ldots,T_{N_0(t)}, N_0(t)$, 
	\begin{align}
	(\lambda_\epsilon(T_1),\lambda_\epsilon(T_2),\ldots,\lambda_\epsilon(T_{N_0(t)}))  \Rightarrow(f(X_1(\infty)),f(X_2(\infty)),\ldots,f(X_{N_0(t)}(\infty)))\label{eq:2.3}
	\end{align}
	as $\epsilon\downarrow 0$. 
	
	The proof of Theorem 3 establishes that $\E(\int_{0}^{t}\tilde{\lambda}_\epsilon(s)ds)^2\rightarrow 0$ as $\epsilon\downarrow 0$; see (\ref{eq:3.7}). Chebyshev's inequality threfore implies that 
	\begin{align} 
	\int_{0}^{t}\tilde{\lambda}_\epsilon(s)ds \Rightarrow 0\label{eq:2.4}
	\end{align}
	as $\epsilon\downarrow 0$. Relations (\ref{eq:2.3}) and (\ref{eq:2.4}) yield the conclusion that
	\[
\exp(-\int_{0}^{t}\tilde{\lambda}_\epsilon(s)ds) \prod_{j=1}^{N_0(t)}{\left(\frac{\lambda_\epsilon (T_j)}{\lambda^*}\right)} \Rightarrow \prod_{j=1}^{N_0(t)}\left(\frac{f(X_j(\infty))}{\E f(X_1(\infty))}\right)
\]
as $\epsilon\downarrow 0$. 

Finally, 
\[ \Big|
\exp(-\int_{0}^{t}\tilde{\lambda}_\epsilon(s)ds) \prod_{j=1}^{N_0(t)}{\left(\frac{\lambda_\epsilon (T_j)}{\lambda^*}\right)}-1\Big| \le 1+ \exp({\color{red}}\|f\| t)\left(\frac{\|f\|}{\lambda^*}\right)^{N_0(t)}
\] 
where $\|f\|\triangleq \max\{|f(x):x\in S|\}$, so that the integrand of the right-hand side of (\ref{eq:2.2}) is bounded uniformly in $\epsilon$ by an integrable rv. Consequently, the Dominated Convergence Theorem applies to the right-hand side of (\ref{eq:2.2}), yielding the theorem.
\end{proof}

It is evident that $N_\epsilon$ does not converge to $N_0$ in total variation, due to the rapid fluctuations in the intensity $\lambda_\epsilon$ at any $\epsilon>0$. However, these rapid fluctuations are ``smoothed out" by path functionals that are suitably continuous, yielding the weak convergence associated with Theorem 1.

\section{An Asymptotic Refinement for the Distribution of $N_\epsilon(t)$}
In this section, we show how the approximation of Theorem 1 can be improved via a ``first order" refinement that reflects the impact of the high frequency fluctuations. Recall that $o(a(\epsilon))$ represents a function of $\epsilon$ such that $o(a(\epsilon))/(a(\epsilon))\rightarrow 0$ as $\epsilon\downarrow 0$. Also, for a bounded (measurable) function on $S$, note that $v$-geometric ergodicity guarantees that if $f_c(x) = f(x) - \E f(X(\infty))$, then

\begin{equation}
|\E_x f_c(X(t))|\le \|f\| d \, v(x) e^{-\alpha t} \label{eq:3.s}
\end{equation}

and hence the integral defining 
\[
g(x)\triangleq \int_{0}^{\infty}\E_x f_c(X(t))dt
\]
converges absolutely and is bounded by a multiple of $v$.

\begin{theorem}
	Suppose Assumption 2 holds and $f$ is bounded (and measurable) with $\E f(X(s))>0$. If $\lambda_\epsilon(t)=f(X(t/\epsilon))$, then 
	\begin{align*}
	P(N_\epsilon(t)=k) = \,&P(N_0(t)=k) \\
	&+ \epsilon P(N_0(t)=k) \left[ \left(\frac{k}{\lambda^*t}-1\right) g(x) +  {\frac{1}{2}} 
	\left(1-\frac{2k}{\lambda^* t} + 
	\frac{k(k-1)}{(\lambda^* t)^2}\right)\sigma^2 t\right] + o(\epsilon)
	\end{align*}
	as $\epsilon\downarrow0$, where $\sigma^2 =2\E f_c(X(\infty))g(X(\infty))$.
\end{theorem}
\begin{proof}
	If we condition on $X$, we find that 
	\[
	P_x(N_\epsilon(t) = k) = \E_x \exp\left(-\int_{0}^{t}\lambda_\epsilon(s)ds\right) \frac{(\int_{0}^{t}\lambda_\epsilon(s)ds)^k}{k!}.
	\]
	Set $h_k(y) = e^{-y} y^k/k!$, and note that for $y>0$, 
	\begin{align*}
	& h_k^{(1)}(y) = h_k(y) \left(\frac{k}{y}-1\right),\\
	& h_k^{(2)}(y) = h_k(y)\left(1-\frac{2k}{y}+\frac{k(k-1)}{y^2}\right),\\
	& h_k^{(3)}(y) = h_k(y) \left(\frac{k(k-1)(k-2)}{y^3} -\frac{3k(k-1)}{y^2}+\frac{3k}{y}-1\right).
	\end{align*}
	Hence, a Taylor expansion of $h_k$ about $t\E f(X(\infty))$ implies that
	\begin{align}
	h_k\left(\int_{0}^{t}\lambda_\epsilon(s)ds\right) &=h_k\left(\epsilon\int_{0}^{t/\epsilon}f(X(s))ds\right) \nonumber\\
	&= h_k\left(t\E f(X(\infty))\right) + h_k^{(1)}(t\E f(X(\infty))) \left(\epsilon\int_{0}^{t/\epsilon}f_c(X(s))ds\right) \nonumber\\
	&\quad + \frac{h_k^{(2)}(t\E f(X(\infty)))}{2} \left(\epsilon\int_{0}^{t/\epsilon}f_c(X(s))ds\right)^2
	+ \frac{h_k^{(3)}(\xi(\epsilon))}{6} \left(\epsilon\int_{0}^{t/\epsilon}f_c(X(s))ds\right)^3,\label{eq:3.1}
	\end{align}
	where $\xi(\epsilon)$ lies between $\int_{0}^{t}\lambda_\epsilon(s)ds$ and $t\E f(X(\infty))$.
	
	Note that (\ref{eq:3.s}) implies that 
	\begin{equation}
	\E_x \int_{0}^{t/\epsilon}f_c(X(s))ds = \int_{0}^{t/\epsilon}\E_x f_c(X(s))ds = g(x) + o(1)v(x) \label{eq:3.1a}
	\end{equation}
	as $\epsilon\downarrow 0$. Also, the Markov property implies that 
	\begin{align}
	&\epsilon\E_x\left(\int_{0}^{t/\epsilon}f_c(X(s))ds\right)^2 \nonumber\\
	&= 2\epsilon \int_{0}^{t/\epsilon}\int_{s}^{t/\epsilon} \E_x f_c(X(s)) f_c(X(u))duds  \nonumber\\
	&= 2\epsilon\int_{0}^{t/\epsilon}\E_x f_c(X(s))\int_{0}^{\infty}\E_x[f_c(X(s+u))|X(s)]duds  \nonumber\\
	&\quad - 2\epsilon\int_{0}^{t/\epsilon}\E_x f_c(X(s))\int_{0}^{\infty}\E_x[\E_x[f_c(X(t/\epsilon+u))|X(t/\epsilon)]|X(s)]duds \nonumber \\
	&= 2\epsilon\int_{0}^{t/\epsilon}\E_x f_c(X(s))g(X(s))ds - 2\epsilon \int_{0}^{t/\epsilon} \E_x f_c(X(s))g(X(t/\epsilon))ds. \label{eq:3.2}
	\end{align}
	
	Because $f$ is bounded and $g$ is bounded by a multiple of $v$, it follows that $fg$ is bounded by a multiple of $v$, so that (\ref{eq:2.1}) implies that 
	\begin{align}
	\epsilon\int_{0}^{t/\epsilon}\E_x f_c(X(s))g(X(s))ds = t \E f_c(X(\infty))g(X(\infty)) + o(1) \label{eq:3.3}
	\end{align}
	as $\epsilon\downarrow 0.$
	Also, 
	\begin{align}
	&\epsilon\int_{0}^{t/\epsilon}\E_x f_c(X(s))g(X(t/\epsilon)) ds \nonumber\\
	=\,& \epsilon\int_{0}^{t/\epsilon - \epsilon^{-1/2}} \E_x f_c(X(s))\E_x[g(X(t/\epsilon))|X(s)] ds + \epsilon\,\E_x\int_{t/\epsilon-\epsilon^{-1/2}}^{t/\epsilon} f_c(X(s))g(X(t/\epsilon))ds. \label{eq:3.4}
	\end{align}
	Since $\E g(X(\infty))=0$, (\ref{eq:2.1}) implies that 
	\[
	|\E_x[g(X(t/\epsilon))|X(s)]| \le \|f\| d \, e^{-\alpha(t/\epsilon-s)} v(X(s)),
	\]
	so that 
	\begin{align}
	& \Big|\epsilon\int_{0}^{t/\epsilon - \epsilon^{-1/2}} \E_x f_c(X(s))\E_x[g(X(t/\epsilon))|X(s)] ds\Big| \nonumber\\
	\le\,& \|f\|^2 d\,\epsilon e^{-\alpha\epsilon^{-1/2}}\int_{0}^{t/\epsilon} \E_x v(X(s))ds \nonumber\\
	=\,& \|f\|^2 d\, e^{-\alpha\epsilon^{-1/2}}\E v(X(\infty)) + o(1)v(x)  \nonumber\\
	=\,& o(1)v(x) \label{eq:3.5}
	\end{align}
	as $\epsilon\downarrow0$. Furthermore, (\ref{eq:2.1}) and the boundedness of $f$ ensure that 
	\begin{equation}
	\Big|\epsilon\E_x\int_{t/\epsilon-\epsilon^{-1/2}}^{t/\epsilon} f_c(X(s))g(X(t/\epsilon))ds\Big|\le \epsilon^{\frac{1}{2}} \|f\| \E_x g(X(t/\epsilon)) = o(1) v(x) \label{eq:3.6}
	\end{equation}
	as $\epsilon\downarrow 0$, and consequently, (\ref{eq:3.2}) through (\ref{eq:3.6}) yield 
	\begin{align}
	\epsilon\,\E_x \Big(\int_{0}^{t/\epsilon}f_c(X(s))ds\Big)^2 = 2t \,{ \E f_c(X(\infty))g(X(\infty)) }+ o(1)v(x)  \label{eq:3.7}
	\end{align}
	as $\epsilon\downarrow 0$.
	
	Finally, note that for $y\ge 0$,
	\begin{align*}
	|h_k^{(3)}(y)| &= \big|\frac{1}{k!}e^{-y}y^{k-3}\big|[-y^3 + 3k y^2-3k(k-1)y +k(k-1)(k-2)]\\
	&\le \frac{(y\lor 1)^k}{k!}(1+3k+3k(k-1)+k(k-1)(k-2))\\
	&\le \frac{8(y\lor 1)^k}{(k-3)!} I(k\ge 3) + 8(y\lor 1)^k I(k\le 2),
	\end{align*}
	where $y\lor 1\triangleq \max(y,1)$. Since $f$ is bounded, it is evident that $h^{(3)}(\xi(\epsilon))$ is a bounded rv. Given (\ref{eq:3.1}), our theorem follows if we prove that 
	\begin{equation}
	\epsilon^2 \E_x \left(\int_{0}^{t/\epsilon}f_c(X(s))ds\right)^3 = o(1) \label{eq:3.ss}
	\end{equation}
	as $\epsilon\downarrow 0$. But (\ref{eq:3.7}) implies that 
	\begin{align}
	& \epsilon^2 \E_x \left(\int_{0}^{t/\epsilon}f_c(X(s))ds\right)^3 \nonumber\\
	=\, & 6 \epsilon^2 \int_{0}^{t/\epsilon-\epsilon^{-1/2}}\E_x f_c(X(s_1)) \int_{s_1}^{t/\epsilon}f_c(X(s_2))\int_{s_2}^{t/\epsilon} f_c(X(s_3))\,ds_3ds_2ds_1 \nonumber\\
	&\, + \epsilon^2 \E_x \left(\int_{t/\epsilon-\epsilon^{-1/2}}^{t/\epsilon}f_c(X(s))ds\right)^3 \nonumber\\
	=\, & 6\epsilon \int_{0}^{t/\epsilon-\epsilon^{-1/2}} \E_x f_c(X(s)) [(t - \epsilon s)\E f_c(X(\infty))g(X(\infty)) +  \epsilon o(1)v(X(s))] ds \nonumber \\ 
	&\, + \epsilon^2 \E_x \left(\int_{t/\epsilon-\epsilon^{-1/2}}^{t/\epsilon}f_c(X(s))ds\right)^3,
	\label{eq:3.15z}
	\end{align}		
	where the term $o(1)$ holds uniformly over $0\le s\le t/\epsilon - \epsilon^{-1/2}$. The boundedness of $f$ implies that 
	\begin{equation}
	\epsilon^2 \E_x \left(\int_{t/\epsilon-\epsilon^{-1/2}}^{t/\epsilon}f_c(X(s))ds\right)^3 \le \epsilon^{1/2}\|f\|^3 \rightarrow 0 \label{eq:3.16z}
	\end{equation}
	as $\epsilon\downarrow 0$. On the other hand, 
	(\ref{eq:3.s}) implies that 
	\begin{align*}
	\int_{0}^{\infty} |\E_x f_c(X(s))|(1+s)ds <\infty,
	\end{align*}
	so we conclude that 
	\begin{equation}
	\epsilon\int_{0}^{t/\epsilon-\epsilon^{-1/2}} |\E_x f_c(X(s))| (t-\epsilon s) ds \rightarrow 0 \label{eq:3.17z}
	\end{equation}
	as $\epsilon\downarrow 0$. Also, 
	\begin{align*}
	 \Big|\epsilon\int_{0}^{t/\epsilon-\epsilon^{-1/2}} o(1)\E_x v(X(s))f_c(X(s))  ds\Big| &\le o(1) \|f\| \epsilon \int_{0}^{t/\epsilon} \E_x v(X(s))ds \\
	&=o(1) \|f\| t\E v(X(\infty)) (1+o(1))\rightarrow 0
	\end{align*}
	as $\epsilon\downarrow 0$, proving (\ref{eq:3.ss}) in view of (\ref{eq:3.15z}), (\ref{eq:3.16z}), and (\ref{eq:3.17z}), and thereby establishing the theorem. \end{proof}

	A similar (but easier) calculation follows in the deterministic periodic setting in which $\lambda(\cdot)$ is deterministic with period $1$, say. In this case,
	\begin{equation*}
	P(N_\epsilon(t) = k) = P(N_0(t) = k) \left(1+\epsilon\left(\frac{k}{\lambda^* t}-1 \right) \int_{\lfloor t/\epsilon\rfloor}^{t/\epsilon}(\lambda(s)-\lambda^*)ds +o(\epsilon) \right)
	\end{equation*}
	as $\epsilon\downarrow 0$, where $\lambda^* = \int_{0}^{t}\lambda(r)dr$ and $\lfloor x\rfloor$ denotes the greatest integer less than or equal to $x$.

\section{An Asymptotic Refinement for Infinite Server Queues}
In this section, we study our Poisson approximation (and its associated first order ``error correction") in the setting of the infinite-server queue. Assume that the system starts empty at $t=0$, and that the service times $V_1,V_2,\ldots$ assigned to arriving consecutive customers are iid and independent of $N_\epsilon$. Our goal in this section is to study the number-in-system process $Q_\epsilon = (Q_\epsilon(t):t \ge 0)$, when $Q_\epsilon$ has arrival process $N_\epsilon$ and service time sequence $V=(V_n:n\ge 1)$. Let $Q_0 = (Q_0(t):t\ge 0)$ be the number-in-system process associated with the constant rate Poisson process $N_0$ and the same service time sequence $V$. Our main result in this section is our next theorem.  
\begin{theorem}
	Assume Assumption 2 and suppose that $f$ is bounded (and measurable) with $\E f(X(\infty))>0$. Suppose $V_1$ has a density $k = (k(x):x\ge 0)$, and set $\bar{K}(x)\triangleq P(V_1>x)$. If $\lambda_\epsilon(t) = f(X(t/\epsilon))$, then 
	\begin{align*}
	P(Q_\epsilon(t) = k)  = P(Q_0(t) = k)\left(1+\epsilon\left[ \left(\frac{k}{\E Q_0(t) } - 1\right)g(x) \bar{K}(t) + {\frac{1}{2}}\left(1-\frac{2k}{\E Q_0(t)} + \frac{k(k-1)}{(\E Q_0(t))^2}\right)\eta^2  \right] + o(\epsilon)\right),
	\end{align*}
	where $\eta^2 = 2 \sigma^2 \int_{0}^{t}\bar{K}(s)k(s) s\,ds + \sigma^2 t\bar{K}(t)^2$, and $\sigma^2$ and $g$ are as in Section 3.
\end{theorem}
\begin{proof}
	The argument closely follows that of Theorem 3. Because $Q_\epsilon(t)$ is, conditional on $X$, Poisson distributed (see \cite*{massey1993networks}), it follows that 
	\begin{align*}
	P(Q_\epsilon(t) = k) = \E_x \exp\left(-\int_{0}^{t}\lambda_\epsilon(s)\bar{K}(t-s)ds\right)\cdot \left(\int_{0}^{t}\lambda_\epsilon(s)\bar{K}(t-s) ds\right)^k \frac{1}{k!}.
	\end{align*}
	As in Theorem 3, we now Taylor expand $h_k(\cdot)$. In this setting, we expand about $\E Q_0(t)$. It follows that the first order term here is $h^{(1)}_k(\E Q_0(t))$ multiplied by 
	\begin{align*}
	& \int_{0}^{t}[\lambda_\epsilon(s)-\lambda^*]\bar{K}(t-s)ds \\
	=\, &\int_{0}^{t}f_c(X(s/\epsilon)) \int_{t-s}^{\infty} k(u) du\,ds \\
	=\, & \int_{0}^{\infty}\int_{0}^{t} I(s>t-u) f_c(X(s/\epsilon))ds \,k(u) du\\
	=\,& \int_{0}^{\infty} k(u)[\epsilon A_c(t/\epsilon) - \epsilon A_c((t-u)/\epsilon)]du,
	\end{align*}
	where $A_c(r) = 0$ for $r\le 0$ and $A_c(r) = \int_{0}^{r} f_c(X(s))ds $ for $r\ge 0$. We note that 
	\[
	\E_x A(t/\epsilon) -\E_x A((t-u)/\epsilon) \rightarrow
	\begin{cases}
	0,& 0\le u\le t\\
	g(x),              & u>t
	\end{cases}
	\]
	as $\epsilon\downarrow 0$, uniformly in $u\le t-\sqrt{\epsilon}$. Accordingly, 
	\[
	\epsilon \E_x \int_{0}^{t}[\lambda_\epsilon(s)-\lambda^*]\bar{K}(t-s)ds =\epsilon g(x)\bar{K}(t) (1+o(1))
	\]
	as $\epsilon\downarrow 0$. 
	
	As for the second derivative term, we are led to the consideration of 
	\begin{align}
	& \epsilon \E_x \left(\int_{0}^{\infty}k(u)[A_c(t/\epsilon)-A_c((t-u)/\epsilon)]du\right)^2 \nonumber\\
	=\,& {2}\epsilon \int_{0}^{\infty}k(u_1)\int_{u_1}^{\infty} k(u_2) \E_x[(A_c(t/\epsilon)-A_c((t-u_1)/\epsilon)(A_c(t/\epsilon)-A_c((t-u_2)/\epsilon)] du_2 du_1. \label{eq:4.1}
	\end{align}	
	Note that for $0\le u_1\le u_2\le t$, 
	\begin{align}
	& \epsilon\E_x(A_c(t/\epsilon)-A_c((t-u_1)/\epsilon)( A_c((t-u_1)/\epsilon) - A_c((t-u_2)/\epsilon )) \nonumber\\
	=\, &  \epsilon\int_{(t-u_2)/\epsilon}^{(t-u_1)/\epsilon} \int_{(t-u_1)/\epsilon}^{t/\epsilon} I(|s_1-s_2|\le \epsilon^{-1/2}) \E_x f_c(X(s_1))f_c(X(s_2))ds_1 ds_2 \nonumber \\
	& +\epsilon \int_{(t-u_2)/\epsilon}^{(t-u_1)/\epsilon} \int_{(t-u_1)/\epsilon}^{t/\epsilon} I(|s_1-s_2|> \epsilon^{-1/2}) \E_x f_c(X(s_1))f_c(X(s_2))ds_1 ds_2. \label{eq:4.1a}
	\end{align}
	The first term on the right-hand side of (\ref{eq:4.1a}) can be upper bounded by 
	\[
	\epsilon\|f\|^2 \int_{(t-u_2)/\epsilon}^{(t-u_1)/\epsilon} \int_{(t-u_1)/\epsilon}^{t/\epsilon} I(|s_1-s_2|\le \epsilon^{-1/2}) ds_1 ds_2 = O(\epsilon^{1/2})\rightarrow0 
	\]
	as $\epsilon\downarrow 0$. For the second term, we use (\ref{eq:3.s}) to obtain the upper bound 
	\begin{align*}
	&\epsilon\int_{(t-u_2)/\epsilon}^{(t-u_1)/\epsilon} \int_{(t-u_1)/\epsilon}^{t/\epsilon} I(|s_1-s_2|> \epsilon^{-1/2}) \E_x f_c(X(s_2)) O(\E_x v(X(s_2)))e^{-\alpha(s_1-s_2)} ds_1 ds_2 \\
	&\, \le \epsilon e^{-\alpha \epsilon^{-1/2}} \|f\|\int_{(t-u_2)/\epsilon}^{(t-u_1)/\epsilon} \int_{(t-u_1)/\epsilon}^{t/\epsilon} O(\E_x v(X(s_2)))ds_1 ds_2\rightarrow 0 
	\end{align*} 
	as $\epsilon\downarrow 0$. Consequently, (\ref{eq:4.1}) equals 
	\begin{align*}
	{{2}\,}\epsilon\int_{0}^{\infty}k(u_1)\int_{u_1}^{\infty}k(u_2) \E_x (A_c(t/\epsilon) - A_c((t-u_1)/\epsilon))^2 du_2du_1 + o(1)
	\end{align*}
	as $\epsilon\downarrow 0$. But (\ref{eq:3.7}) proves that 
	\[
	\epsilon\E_x (A_c(t/\epsilon) -A_c((t-u)/\epsilon))^2 \rightarrow 
	\begin{cases}
	\sigma^2 u,& 0\le u\le t\\
	\sigma^2 t,              & u>t
	\end{cases}
	\]
	uniformly in $0\le u\le t$. As a consequence, (\ref{eq:4.1}) equals $\eta^2 + o(1)$ as $\epsilon\downarrow 0$. 
	
	The third derivative term can be handled similarly as in Theorem 3, thereby yielding the proof of the result. 
\end{proof}


\end{document}